\documentclass{amsart}

\usepackage{amsmath, amsthm, amssymb}
\usepackage{epsfig}
\usepackage{array}
\usepackage{amsmath}
\usepackage{amsfonts}
\usepackage{amsfonts}
\usepackage{amsmath}%
\usepackage{amsfonts}%
\usepackage{amssymb}%
\usepackage{graphicx}

\usepackage{dsfont}

\usepackage{amsfonts}
\usepackage{amssymb}
\usepackage{amsmath,amsxtra,amssymb}

\numberwithin{equation}{section}

\newtheorem{theorem}{Theorem}
\newtheorem{lemma}[theorem]{Lemma}

\newtheorem{corollary}[theorem]{Corollary}

\newcommand{\field}[1]{\mathbb{#1}}

\newcommand{\la}{{\langle}}

\newcommand{\om}{{\omega}}

\newcommand{\Cr}{{C^*_r(G)}}
\newcommand{\Br}{{B_r(G)}}

\newcommand{\abs}[1]{|#1|}
\newcommand{\norm}[1]{\|#1\|}

\newcommand{\G}{\mathbb G}

\def\proclaim #1. #2\par{\medbreak
\noindent{\bf#1.\enspace}{\sl#2}\par\medbreak}
\allowdisplaybreaks

\title{On iterated powers of positive definite functions}
\author{Mehrdad Kalantar}
\address{ School of Mathematics and Statistics,
        Carleton University, Ottawa, Ontario, Canada K1S 5B6}
\email{mkalanta@math.carleton.ca}


\begin{document}
\begin{abstract}
We prove that if $\rho$ is an adapted positive definite
function in the Fourier--Stieltjes algebra $B(G)$ of a locally compact group $G$ with $\|\rho\|_{B(G)}=1$,
then the iterated powers $(\rho^n)$ converge to zero
in the weak* topology $\sigma(B(G) , C^*(G))$.
Moreover, if $\rho$ is irreducible, we prove that $(\rho^n)$
as a sequence of u.c.p. maps on the group $C^*$-algebra converges to zero in the strong operator topology.
\end{abstract}


\maketitle




In this paper we prove dual versions of two fundamental limit theorems of convolution powers
of probability measures on locally compact groups:

\begin{theorem}\cite[Theorem 2]{Mukh} \and \cite[TH\'{E}OR\`{E}ME 8]{Derr} \label{Muk}
Let $\mu$ be a probability measure on a locally compact group $G$.
Suppose $G$ is not compact and $\mu$ is adapted (i.e. support of $\mu$ generates $G$ as a closed subgroup),
then the convolution powers $(\mu^n)$ converge to zero in the weak* topology $\sigma(M(G) , C_0(G))$.
\end{theorem}

\begin{theorem}\cite[Corollary 4]{HM} \and \cite[Theorem 1.8]{JRW} \label{jrw}
Let $\mu$ be a probability measure on a locally compact group $G$.
Suppose $G$ is not compact and $\mu$ is irreducible (i.e. support of $\mu$ generates $G$ as a closed semigroup),
then $\|\mu^n\star f\|_\infty$ converges to zero for every $f \in C_0(G)$.
\end{theorem}

These results have a number of important consequences in the study of harmonic functions and boundaries of random 
walks on locally compact groups.

The latter, also known as the {\it concentration function problem} for locally compact groups
was first considered by Hofmann and Mukherjea in \cite{HM}, where they proved the result
for a large class of locally compact groups, but left open the general case.
It was then in \cite{JRW}, where Jaworski, Rosenblatt, and Willis used the developments
in the theory of totally disconnected groups to settle the problem in the general
locally compact groups case.

A dual version of the theory of random walks on groups, harmonic functions, and measure-theoretic boundaries has been
developed by Biane \cite{Biane} and Chu--Lau \cite{Chu-Lau} (see also \cite{I}).
The noncommutative versions of the above theorems have attracted interests of many in the area.

But the proofs of the above theorems, specially Theorem \ref{jrw},
are based on some very deep results on the structure theory of
locally compact groups that cannot be modified in the noncommutative world!
So, one has to provide new arguments in other settings.

Here we prove these results in the dual setting,
i.e. for positive definite functions in the Fourier--Stieltjes algebra $B(G)$ of $G$.

Note that the author has proved a discrete quantum group version of Theorem \ref{jrw} in \cite{K-JFA}.
Of course, the dual of a locally compact group is a locally compact quantum group,
but the proof of \cite{K-JFA} which uses Banach limits and a non-commutative 0-2 law
cannot be generalized to general (quantum) probability measures
on non-discrete locally compact quantum groups.

So, we present a different proof for the co-commutative case, i.e. duals of locally compact groups.

Recall that for a locally compact group $G$ the set $B(G)$ of all matrix coefficient functions
of continuous unitary representations of $G$ forms a subalgebra of bounded
continuous functions on $G$. It admits a norm $\|\cdot\|_{B(G)}$ with which it turns to a
Banach algebra called the Fourier--Stieltjes algebra, which is isomorphic to the dual Banach space of $C^*(G)$
(the universal group $C^*$-algebra of $G$).

For abelian $G$ with (Pontryagin) dual group $\hat G$, the 
Fourier--Stieltjes transform $\mathcal{F}_s$ yields an isomorphism
of the dual Banach algebras $B(G) \cong M(\hat G)$,
where $(M(\hat G) , \star)$ is the measure algebra (with convolution product) of the dual group $\hat G$.

We denote by $P_1(G)$ the set of positive definite functions of norm one. So,
$\rho\in P_1(G)$ means that there
exists a continuous unitary representation $\pi$ of $G$ on a Hilbert space $H_\pi$,
and a unit vector $\xi_\rho \in H_\pi$ such that
\begin{equation}\label{002}
\rho(r)\, =\, \langle\, \pi(r)\, \xi_\rho\, ,\, \xi_\rho \, \rangle \hspace{2cm}  (r\,\in\, G)\,.
\end{equation}

Following \cite{NR}, $\rho$ is said to be {\it adapted} if $G_\rho := \rho^{-1}(\{1\}) = \{e\}$.

\begin{theorem}\label{w*}
Let $G$ be a non-discrete locally compact group, and let $\rho\in P_1(G)$ be adapted.
Then $(\rho^n)$ converge to zero in the weak* topology $\sigma(B(G) , C^*(G))$.
\end{theorem}
\begin{proof}
Let $G_{\bar\rho} = \rho^{-1}(\mathbb T)$, where $\mathbb T$ is the unit circle in the complex plane.
Then it is easily seen from the representation (\ref{002})
that $G_{\bar\rho}$ is a closed subgroup of $G$, and the restriction
of $\rho$ to $G_{\bar\rho}$ is a group homomorphism. And since $\rho$ is adapted, this restriction
is in fact injective.

Now, towards a contradiction, suppose that there exists a subnet
$\left(\rho^{n_i^0}\right)_{i\in I_0}$ of $\left(\rho^n\right)$
that converges weak* to a non-zero $\nu_0 \in B(G)^+$. Then for each $k\in \mathbb N$, find inductively a subnet
$\left(\rho^{n_i^k}\right)_{i\in I_k}$ of $\left(\rho^{n_i^{k-1}}\right)_{i\in I_{k-1}}$ such that
$\left(\rho^{n_i^k - k}\right)_{i\in I_k}$ converges. Denote
\[
\nu_k \,:=\, \lim_i\,\rho^{n_i^k - k}\, \in\, B(G)^+\,.
\]
Thus, by construction, we see that
\begin{equation}\label{001}
\rho^m \nu_{k+m-1}\, =\, \nu_{k-1} \ \ \ \ \ \ \ \ k,\,m\,\in\,\mathbb{N}
\end{equation}
which implies that $0 < \|\nu_0\|_{B(G)} = \|\nu_k\|_{B(G)} \leq 1$, and support$(\nu_{k-1})$ $\subseteq$
support$(\rho)$ $\cap$ support$(\nu_{k+m-1})$ for all $k, m\in \mathbb N$.
Moreover, since $\|\rho\|_\infty = 1$ it follows from (\ref{001}) that
\[
|\nu_k(r)| = |\rho(r)|^m\, |\nu_{k+m}(r)| \leq |\rho(r)|^m \longrightarrow 0
\]
for all $r\in G - G_{\bar\rho}$.
Hence, support$(\nu_k)$ $\subseteq$ $\G_{\bar\rho}$ for all $k\geq 0$.
From continuity of $\nu_k$'s we conclude that $G_{\bar\rho}$ is open in $G$.
So, we have a canonical identification
\[
B(G_{\bar\rho})\, \cong \, \{\, \mathds{1}_{G_{\bar\rho}}\,\cdot\, \om\, :\, \om\, \in\, B(G)\,\}\,,
\]
where $\mathds{1}_{G_{\bar\rho}} \in B(G)$ is the characteristic function of $G_{\bar\rho}$ (cf. \cite{Ey}).
Through this identification, the weak* topology $\sigma(B(G_{\bar\rho}) , C^*(G_{\bar\rho}))$ coincides with
the restriction of the $\sigma(B(G) , C^*(G))$-topology.
Now, since $G_{\bar\rho}$ is abelian, the Fourier--Stieltjes transform $\mathcal{F}_s$ induces the identification
$B(G_{\bar\rho}) \cong M(\widehat{G_{\bar\rho}})$.
Then $\mathcal{F}_s({\mathds{1}_{G_{\bar\rho}}\, \rho}) \in M(\widehat{G_{\bar\rho}})$
is an adapted probability measure \cite[Proposition 2.1]{NR}. Moreover, as an open subgroup
of a non-discrete group, $G_{\bar\rho}$ is not discrete, i.e. $\widehat{G_{\bar\rho}}$
is not compact. Hence, by Theorem \ref{Muk} the convolution powers
$\big(\mathcal{F}_s({\mathds{1}_{G_{\bar\rho}}\, \rho})\big)^{\star n}$
converge to 0 in the weak* topology of $M(\widehat{G_{\bar\rho}})$. This gives
\[
\mathds{1}_{G_{\bar\rho}}\, \rho^n\, =\,
\mathcal{F}_s^{-1}\big(\mathcal{F}_s({\mathds{1}_{G_{\bar\rho}}\, \rho})^{\star n}\big)\,\, {\longrightarrow}\,\,0
\]
in the $\sigma(C^*(G) , B(G))-$topology. Consequently, we have
\[
\nu_0 \,=\, \mathds{1}_{G_{\bar\rho}}\, \nu_0 \,=\, \mathds{1}_{G_{\bar\rho}} \,\lim_i\,\rho^{n_i^0}
\,=\, \lim_i\, \mathds{1}_{G_{\bar\rho}} \,\rho^{n_i^0} \, =\,0
\]
which is a contradiction.
\end{proof}

\noindent
{\bf Note.} After the completion of this work it was pointed out to the author that Theorem \ref{w*} also follows from a more
general result proved in \cite[Theorem 5.3]{Kan-Ulg}.\\

Since in the abelian case the reduced $C^*$-algebra $\Cr$ coincides
with $C^*(G)$, one may consider $\Br = \Cr^*$ as the dual object to the measure algebra $M(G)$.
So, using the fact that $\Cr$ is a quotient $C^*$-algebra of $C^*(G)$, and that $\Br$ is
a subalgebra of $B(G)$ we derive the following.

\begin{corollary}
Let $G$ be a non-discrete locally compact group, and let $\rho\in B_r(G)$ be an adapted positive definite function on $G$.
Then the iterated powers $(\rho^n)$ converge to zero in the weak* topology $\sigma(B_r(G) , C^*_r(G))$.
\end{corollary}

Next, we prove a stronger limit theorem for irreducible positive definite functions.
$\rho\in P_1(G)$ is \emph{irreducible}
if for every non-zero positive $x\in  C(G)^*$
there exists $n\in\field{N}$ such that
$\langle\, x\, ,\, \mu^n\,\rangle\, \neq\, 0$.

In the case of an abelian group $G$, this class consists of the probability measures on the dual group $\hat G$
that the smallest closed semigroup containing their support is $\hat G$.

\begin{lemma}
Every irreducible positive definite function is adapted.
\end{lemma}
\begin{proof}
This follows from \cite[Proposition 3.2.10]{Chu-Lau} and \cite[Theorem 3.6]{KNR}.
\end{proof}

\begin{lemma}\label{der}
Suppose that $\mathfrak{A}$ is a $C^*$-algebra, and $T$ is a positive contraction on $\mathfrak{A}$.
Then for any $0 \leq x \in \mathfrak{A}$ we have
\[
\lim_n\, \norm{\,T^n x\,}\,=\, \sup\,\{\,|\,\langle\,x\,,\,\nu\,\rangle\,| \,:\, \nu\,\in \mathcal{D}_1\,\}
\]
where $\mathcal{D}_1$ denotes the set of all $0 \leq \nu\in \mathfrak{A}^*$ such that there exists
a sequence $\displaystyle \{\nu_n\}_{n=0}^\infty$ of positive elements of
$\mathfrak{A}^*$ such that $\nu_0 = \nu$, $\|\nu_n\| \leq 1$
for all $n\geq 0$, and
\begin{equation}\label{22}
T^*\,\nu_{n+1}\,=\,\nu_n \hspace{2cm} n\geq 0\,.
\end{equation}
\end{lemma}

\begin{proof}
The ``$\geq$'' part of the equation follows from $(T^*)^n \nu_{n} = \nu_0$.
For the reverse inequality, let $0\leq x\in\mathfrak{A}$, and $\varepsilon > 0$.
Choose $0\leq \om_n \in \mathfrak{A}^*$ with $\norm{\om_n} \leq 1$
such that $\norm{\,T^n x\,}\, < \, |\,\langle\,T^n x\,,\,\om_n\,\rangle\,|\, +\, \varepsilon$ for all $n$.
Then, let $\nu_0$ be a weak* cluster point of $\{\,(T^*)^{n}\om_n\,:\,n\geq 1\,\}$ in $\mathfrak{A}^*$.
So, we may find a subnet $(T^*)^{n_i}\om_{n_i}$ that converges weak* to $\nu_0$,
and therefore
\[
\lim_n\, \norm{\,T^n\,x\,}\,=\, \lim_{n_i}\, \norm{\,T^{n_i}\,x\,}\,\leq\,
\lim_{n_i}\,|\,\langle\, x\,,\,(T^*)^{n_i}\om_{n_i}\,\rangle\,|\, +\, \varepsilon
\,=\, |\,\langle\,x\,,\,\nu_0\,\rangle\,|\, +\, \varepsilon\,.
\]
Now, let $0\leq \eta_1$ be a weak* cluster point of
$\{\,(T^*)^{n_i - 1}\om_{n_i}\,\}$ in $\mathfrak{A}^*$.
Then $\norm{\eta_1}\leq 1$ and $T^* \eta_1 = \nu_0$. Continuing this way, by induction
we can construct a sequence $\{\eta_n\}$ of positive elements of
$\mathfrak{A}^*$ such that $\|\eta_n\| \leq 1$,
and $T^* \eta_{n+1} = \eta_n$ for all $n$. So, we have in particular $(T^*)^n \eta_{n} = \nu_0$.
Now, let $\displaystyle \nu_1 = \lim_{n_j} \, (T^*)^{n_j - 1}\eta_{n_j}$ be a weak* cluster point of
$\{\,(T^*)^{n-1}\eta_n\,:\,n\geq 2\,\}$ in $\mathfrak{A}^*$.
Then $\norm{\nu_1}\leq 1$ and $T^* \nu_1 = \nu_0$. And, if we choose a weak* cluster point $\nu_2$
of $\{\,(T^*)^{n_j-2}\eta_{n_j}\,\}$, then $0\leq \nu_2$, $\norm{\nu_2}\leq 1$, and $T^* \nu_2 = \nu_1$.
Similarly, we can now construct a sequence $\{\nu_n\}$ of positive elements of
$\mathfrak{A}^*$ with $\|\nu_n\| \leq 1$
for all $n\geq 0$ that satisfies (\ref{22}).
\end{proof}

In the following, we denote by $\rho \cdot x$ the canonical action
of elements $\rho$ in $B(G)$ on elements $x$ in $C^*(G) \cong B(G)_*$. 
\begin{theorem}
Let $G$ be a non-discrete locally compact group, and let $\rho\in P_1(G)$ be irreducible.
Then for every $x\in C^*(G)$ we have
\[
\lim_n\, \|\,\rho^n \cdot x\,\|\,=\,0\,.
\]
\end{theorem}
\begin{proof}
Let $\mathcal{D}_1$ be the set of all $\nu\in B(G)$ such that there exists
a sequence $\displaystyle \{\nu_k\}_{k=0}^\infty$ in $B(G)^+$ such that $\nu_0 = \nu$, $\|\nu_k\| \leq 1$, and
$\rho\,\nu_{k+1}\,=\,\nu_k$ for all $k\geq 0$.
Then similarly to the proof of Theorem \ref{w*} 
we can show support$(\nu_k)  \subseteq\G_{\bar\rho}$ for all $k\geq 0$,
and hence $G_{\bar\rho} = \rho^{-1}(\mathbb T)$ is open in $G$.
Again, similarly to the proof of Theorem \ref{w*},
we obtain via the Fourier--Stieltjes transform, the sequence $\{\mathcal{F}_s(\mathds{1}_{G_{\bar\rho}} \nu_k)\}$
in the measure algebra $\big(\, M({G_{\bar\rho}})\,,\,\star \,\big)$
such that $\|\mathcal{F}_s(\mathds{1}_{G_{\bar\rho}} \nu_k)\| \leq 1$, and
\[
\mathcal{F}_s(\mathds{1}_{G_{\bar\rho}} \rho)\, \star \, \mathcal{F}_s(\mathds{1}_{G_{\bar\rho}} \nu_{k+1})
\,=\,
\mathcal{F}_s(\mathds{1}_{G_{\bar\rho}} \rho \, \nu_{k+1})
\,=\,
\mathcal{F}_s(\mathds{1}_{G_{\bar\rho}} \nu_k)
\]
for all $k\geq 0$. Therefore we have
\begin{eqnarray*}
\abs{\,\la\, x \,,\, \nu_0\,\rangle\,} & = &
\abs{\,\la\, \mathds{1}_{G_{\bar\rho}}\, x \,,\, \mathds{1}_{G_{\bar\rho}}\, \nu_0\,\rangle\,}\\ & = &
\abs{\,\la\, \mathcal{F}^{-1}(\mathds{1}_{G_{\bar\rho}}\, x) \,,\, \mathcal{F}^{-1}(\mathds{1}_{G_{\bar\rho}}\, \nu_0) \,\rangle\,}
\\ & \leq &
\lim_n\,\|\,\mathcal{F}^{-1}(\mathds{1}_{G_{\bar\rho}}\,\rho)^{\star n}\,\star\,\mathcal{F}^{-1}(\mathds{1}_{G_{\bar\rho}}\, x)\,\|
\hspace{2cm} \text{(By Lemma \ref{der})}\\
&=& 0 \hspace{7.28cm} \text{(By Theorem \ref{jrw})}\,.
\end{eqnarray*}
Then, once more applying Lemma \ref{der},
this time to the action of $\rho$ as a contraction on $C^*(G)$, we conclude that
\[
\lim_n\, \|\,\rho^n \cdot x\,\|\,=\, \sup\,\{\,\abs{\la\, x \,,\, \nu_0\,\rangle}\,:\,\nu_0\in\mathcal{D}_1\,\}
\,=\,0\,.
\]
\end{proof}

\begin{corollary}
Let $G$ be a non-discrete locally compact group,
and let $\rho\in B_r(G)$ be an irreducible positive definite function on $G$.
Then we have
\[
\lim_n\, \|\,\rho^n \cdot x\,\|\,=\,0
\]
for all $x\in C^*(G)$.
\end{corollary}

\bibliographystyle{plain}

\begin{thebibliography}{99}






\bibitem{Biane} Ph. Biane,
\textit{Quantum random walk on the dual of $SU(n)$},
Probab. Theory Related Fields \textbf{89} (1991), 117--129.




\bibitem {Chu-Lau} C.-H. Chu \and A. T.-M. Lau,
\textit{Harmonic functions on groups and Fourier algebras}, Lecture Notes in Mathematics,
\textbf{1782}, Springer-Verlag, Berlin, 2002.



\bibitem{Derr} Y. Derriennic,
\textit{Lois ``z\'{e}ro ou deux'' pour les processus de Markov.
Applications aux marches al\'{e}atoires},
Ann. Inst. H. Poincar\'{e}	 Sect. B (N.S.) {\bf 12} (1976), no. 2, 111--129.

\bibitem{Ey} P. Eymard,
\textit{L'{}alg\`{e}bra de Fourier d'{u}n groupe localement compact},
Bull. Soc. Math. France \textbf{92} (1964), 181--236.







\bibitem {HM} K. H. Hofmann \and A. Mukherjea,
\textit{Concentration functions and a class of noncompact groups.},
Math. Ann. \textbf{256} (1981), no. 4, 535--548.



\bibitem {I} M. Izumi,
\textit{Non-commutative Poisson boundaries and compact quantum group actions}, 
Adv. Math. \textbf{169} (2002), no. 1, 1--57. 



\bibitem {JRW} W. Jaworski, J. Rosenblatt \and G. Willis,
{\it Concentration functions in locally compact groups},
Math. Ann. {\bf 305} (1996), no. 4, 673--691.



\bibitem {K-JFA} M. Kalantar,
\textit{A limit theorem for discrete quantum groups},
J. Funct. Anal. {\bf 265} (2013), no. 3, 469--473.


\bibitem {KNR} M. Kalantar, M. Neufang \and Z.-J. Ruan,
\textit{Poisson boundaries over locally compact quantum
groups}, Internat. J. Math. {\bf 24} (2013), no. 3, 1350023, 21 pp.



\bibitem {Kan-Ulg} E. Kaniuth, A. \"{U}lger,
\textit{The structure of power bounded elements in Fourier--Stieltjes algebras of locally compact groups},
Bull. Sci. Math. {\bf 137} (2013), no. 1, 45--62.



\bibitem {Mukh} A. Mukherjea,
\textit{Limit theorems for probability measures on non-compact groups and semi-groups},
Z. Wahrscheinlichkeitstheorie und Verw. Gebiete \textbf{33} (1976), no. 4, 273--284.






\bibitem {NR} M. Neufang \and V. Runde, \textit{Harmonic operators: the dual perspective},
Math. Z. \textbf{255} (2007), 669--690.








\end{thebibliography}

\end{document}